\documentclass[11pt,draft,reqno]{amsart}
\usepackage{amsmath,amssymb,amsfonts,amscd}
\usepackage{amscd}
\usepackage{graphics}
\usepackage{dsfont}
\usepackage[all]{xy}
\usepackage{enumerate}
\usepackage{dsfont}
\usepackage{bbold}
\usepackage{mathbbol}
\usepackage{mathrsfs}

\usepackage[all]{xy}
\usepackage{enumerate}

{\catcode`\@=11
\gdef\n@te#1#2{\leavevmode\vadjust{%
 {\setbox\z@\hbox to\z@{\strut#1}%
  \setbox\z@\hbox{\raise\dp\strutbox\box\z@}\ht\z@=\z@\dp\z@=\z@%
  #2\box\z@}}}
\gdef\leftnote#1{\n@te{\hss#1\quad}{}}
\gdef\rightnote#1{\n@te{\quad\kern-\leftskip#1\hss}{\moveright\hsize}}
\gdef\?{\FN@\qumark}
\gdef\qumark{\ifx\next"\DN@"##1"{\leftnote{\rm##1}}\else
 \DN@{\leftnote{\rm??}}\fi{\rm??}\next@}}
\def\mnote#1{\leftnote{\vbox{\hsize=2.5truecm\footnotesize
\noindent #1}}}

\DeclareFontFamily{OT1}{wncyr}{\hyphenchar\font45 }
\DeclareFontShape{OT1}{wncyr}{m}{n}{%
   <5> <6> <7> <8> <9> gen * wncyr
   <10> <10.95> <12> <14.4> <17.28> <20.74>  <24.88>wncyr10}{}
\DeclareFontShape{OT1}{wncyr}{m}{it}{%
   <5> <6> <7> <8> <9> gen * wncyi
   <10> <10.95> <12> <14.4> <17.28> <20.74> <24.88> wncyi10}{}
\DeclareFontShape{OT1}{wncyr}{m}{sc}{%
   <5> <6> <7> <8> <9> <10> <10.95> <12> <14.4>
   <17.28> <20.74> <24.88>wncysc10}{}
\DeclareFontShape{OT1}{wncyr}{b}{n}{%
   <5> <6> <7> <8> <9> gen * wncyb
   <10> <10.95> <12> <14.4> <17.28> <20.74> <24.88>wncyb10}{}
\input cyracc.def
\def\rus{\usefont{OT1}{wncyr}{m}{n}\cyracc\fontsize{9}{11pt}\selectfont}

\DeclareMathSizes{9}{9}{7}{5} 

\theoremstyle{plain}

\newtheorem{theorem}{Theorem}

\newtheorem{corollary}{Corollary}

\newtheorem{example}{Example}
\newtheorem*{exnonumber}{Example}

\theoremstyle{definition}
\newtheorem{definition}{Definition}

\theoremstyle{remark}

\newtheorem{example*}{Example}

\def\Cn{\;{\mathcal C}_n}
\def\An{{\bf A}^n}

\def\bAn{{\mathbf A}\!^n}
\def\bAnn{{\mathbf A}\!^{n-1}}
\def\bAm{{\mathbf A}\!^m}
\def\bAs{{\mathbf A}\!^s}

\def\bA2{{\mathbf A}\!^2}
\def\Ab3{{\mathbf A}\!^3}
\def\bA1{{\mathbf A}\!^1}
\def\Aut{{\rm Aut}_k}

\newcommand{\dss}{\hskip -2mm\rotatebox{68}{\raisebox{-1.8\height}{\mbox{\normalsize -\hskip .1mm-\hskip .1mm-}}}\hskip -.6mm}

\begin{document}

\title[Bass' triangulability problem]
{Bass' triangulability problem}

\author[Vladimir L. Popov]{Vladimir L. Popov}
\address{Steklov Mathematical Institute, Russian Academy of Sciences, Gubkina 8, Moscow 119991, Russia}

\address{National Research University\\ Higher School of Economics, Myas\-nitskaya
20, Moscow 101000,\;Russia}

\email{popovvl@mi.ras.ru}

\thanks{Partially
 supported by
 grants {\rus RFFI}
15-01-02158  and {\rus N{SH}}-2998.2014.1.}



\begin{abstract} Exploring Bass' Triangulability Problem on unipotent algebraic subgroups of the affine Cremona groups, we prove a triangulability criterion,
the existence of nontriangulable connected solvable affine algebraic subgroups of the Cremona groups, and
stable triangulability of such subgroups; in particular,
in the stable range we answer Bass' Triangulability Problem is the affirmative.\;To this end we prove a
theorem on invariant subfields of $1$-extensions.\;We also obtain a ge\-ne\-ral construction of all rationally triangulable subgroups of the Cremona groups and,
as an application, classify rationally triangulable connected one-dimensional uni\-potent affine algebraic subgroups  of the Cremona groups up to conjugacy.
\end{abstract}

\maketitle

\section{Introduction}

We assume given an
algebraically closed field $k$ of arbitrary characteristic which
serves
as
domain of definition
for each of the varieties considered below.\;In this paper, ``variety'' means ``algebraic variety'' and
 it is identified with its set
of
$k$-rational points.

Recall that the {\it Cremona group {\rm(}over $k${\rm)} of rank} $n$ is
the group
\begin{equation*}
{\mathcal C}_n:={\rm Aut}_k\,k(\bAn),
 \end{equation*}
and  ${\rm Aut}_k\,k[\bAn]$ is the {\it affine Cremona group {\rm(}over $k${\rm)} of rank} $n$.\;The group ${\rm Bir}\,\bAn$ of rational self-maps of $\bAn$ (resp.\,the group ${\rm Aut}\,\bAn$)
is identified with ${\mathcal C}_n$ (resp.\,${\rm Aut}_k\,k[\An]$) by means of
the isomorphism $\varphi\mapsto (\varphi^*)^{-1}$.
For $n>1$, we identify $k(\bAnn)$ with the subfield of $k(\bAn)$ by means of the natural embedding
$k(\bAnn) \hookrightarrow k(\bAn)$ determined by the projection
$
\bAn\to\bAnn$, $(\alpha_1,\ldots, \alpha_n)\mapsto (\alpha_1,\ldots, \alpha_{n-1})$.
This makes $k[\bAnn]$ the subalgebra of $k[\bAn]$.\;We put $k({\mathbf A}\!^0)=k[{\mathbf A}\!^0]:=k$.

 Let  
 $x_i\colon \bAn\to k$, $(\alpha_1,\ldots,\alpha_n)\mapsto \alpha_i$, be the $i$th standard  coordinate function.
 We have
 \begin{equation}\label{xi}
 k(\bAn)=k(x_1,\ldots, x_n),\quad k[\bAn]=k[x_1,\ldots, x_n].
 \end{equation}

 The group
$\mathcal C_{n-1}$
is identified with
the subgroup of $\mathcal C_{n}$
by means of the embedding
\begin{equation}\label{embC}
\mathcal C_{n-1}\hookrightarrow {\mathcal C}_{n},\;
\varphi\mapsto \widetilde \varphi,\;\mbox{where
$\widetilde \varphi (x_i):=\varphi(x_i)$ for $i<n$ and $\widetilde \varphi(x_{n}):=x_{n}$.}
\end{equation}
This makes ${\rm Aut}_k\,k[\bAnn]$ the subgroup of ${\rm Aut}_k\,k[\bAn]$.

Although the groups ${\mathcal C}_n$ and ${\rm Aut}_k\,k[\bAn]$
are infinite-dimensional for $n>1$ (see \cite{R64}, \cite{P14}), the analogies between them and algebraic groups catch the eye:\;they have the Zariski topology, algebraic subgroups, tori, roots, the Weyl groups, $\ldots$ (see \cite{P13}, \cite{P13-2} and references therein).
The de Jonqui\`eres subgroup
\begin{equation*}
{\mathcal J}_n:=\{\varphi\in {\rm Aut}_k\,k[\bAn]\mid \varphi(x_i)=\alpha_ix_i+h_i,\;\alpha_i\in k\!^\times, h_i
\in k[{\mathbf A}\!^{i-1}]
\}
\end{equation*}
is viewed by some authors as an
analog  of Borel subgroup for ${\rm Aut}_k\,k[\bAn]$ (see, e.g., \cite{B84}).\;Supporting this viewpoint, \cite[Thm.\,3.1]{P13} implies that every algebraic subgroup of ${\mathcal J}_n$ is affine solvable.\;Having in mind conjugacy of Borel subgroups in every finite dimensional affine algebraic group, one leads to the question
whether it is true that every connected solvable affine algebraic subgroup $G$ of   ${\rm Aut}_k\,k[\bAn]$ is conjugate in ${\rm Aut}_k\,k[\bAn]$ to a subgroup of ${\mathcal J}_n$.\;In particular, Problem III in \cite{B84} asks whether it is true  for unipotent $G$.\;In \cite{B84} Bass answered the latter question in the negative for ${\rm char}\,k=0$, $n=3$, and  $G=k^+$, the one-dimensional additive group.\;In \cite{P87} was then elaborated a simple general method yielding
negative answers
for ${\rm char}\,k=0$, $G=k^+$, and all $n>2$ (this method, in the form of usage of so called replicas, became
 the crucial instrument in the recent studies on infinite transitivity of automorphism groups of algebraic varieties \cite{Ka}).\;Given these developments, Bass formulated for ${\rm char}\, k=0$ the following

\vskip 2mm

\noindent{\bf Bass' Triangulability Problem} (\cite[Question 4]{B84}){\bf.} ``If a unipotent group $G$ acts on $\bAn$, can the action be rationally
triangularized, i.e., can we write $k(x_1,\ldots, x_n) = k( y_1,\ldots, y_n)$ so that each subfield
$k(y_1,\ldots, y_i)$ is $G$-invariant?''

\vskip 2mm

Here we explore this problem in the broader context of connected solvable affine algebraic groups $G$ and arbitrary ${\rm char}\,k$.\;To formulate our results we first introduce two definitions.

\begin{definition}[{{\rm $1$-extensions}}] If a field $K$ 
is a purely transcendental  extension of a field $L$ of the transcendence degree $1$,
then, for brevity, we say that
$K$ is a {\it $1$-extension} of $L$.
\end{definition}

\begin{definition}[{{\rm Rationally triangulable subgroups}}]\label{triang} A subgroup $G$ of $\Cn$ is called {\it rationally triangulable}
 if there is  a flag
\begin{equation}\label{K}
k(\bAn)=:K_n \supset K_{n-1}\supset\cdots \supset K_1\supset K_0:=k
\end{equation}
of $G$-stable subfields of $k(\bAn)$ such that $K_i/K_{i-1}$ is
a $1$-extension
for every $i>0$.\;A
rational action
of an algebraic group on $\bAn$
is called {\it rationally triangulable} if the image
of this group under the homomorphism to
${\mathcal C}_n$ determined by this action is rationally triangulable.
\end{definition}

Now we shall formulate our results.

We start with proving the following Theorem \ref{ext1}; it, in turn, yields Theo\-rem \ref{ext1g} that is 
heavily used in the proofs of our next results.

\begin{theorem}[{{\rm Invariant subfields  of
 $1$-extensions}}]\label{ext1} Let $Q$ be a finitely ge\-ne\-rated field extension of $k$ and let $P$ be a
$1$-extension of $\,Q$.\;Let $G$ be a one-dimensional connected solvable affine algebraic subgroup of $\Aut(P)$ such that $Q$ is $G$-stable.
\begin{enumerate}[\hskip 4.2mm\rm(i)]
\item   If $\;Q^G=Q$, then $P^G=Q^G$.
\item If $\;Q^G\varsubsetneq Q$, then $P^G$ is a $1$-extension of $Q^G$.
\end{enumerate}
\end{theorem}


\begin{theorem}[{{\rm Purity of invariant field extensions}}]\label{ext1g}  Let $Q$ be a finitely ge\-ne\-rated field extension of $k$ and let $P$ be a
$1$-extension of $\,Q$.\;Let $G$ be a connected solvable affine algebraic subgroup of $\Aut(P)$ such that $Q$ is $G$-stable.\;Then one of the following holds:
\begin{enumerate}[\hskip 4.2mm\rm(i)]
\item $P^G= Q^G$; 
\item $P^G$ is a $1$-extension of $Q^G$.
\end{enumerate}
\end{theorem}

Using  
Theorem \ref{ext1g}, we obtain the following triangulability criterion:

\begin{theorem}[{{\rm Triangulability criterion}}]\label{cr} The following properties of a  connected solvable affine algebraic subgroup $G$ of the Cremona group $\Cn$  are equivalent:
\begin{enumerate}[\hskip 4.2mm\rm(i)]
\item $k(\bAn)^G$ is purely transcendental over $k$;
\item $G$ is rationally triangulable.
\end{enumerate}
\end{theorem}

Theorem \ref{cr} generalizes
  \cite[Thm.\,3.1]{DF91}, where the claim is proved
  for one-dimensional unipotent algebraic subgroups of  ${\rm Aut}_k \bAn$ in the case ${\rm char}\,k=0$.

\begin{corollary}[{{\rm Low-dimensional quotients}}]\label{cor1}
A connected solvable affine algebraic subgroup $G$ of $\;\mathcal C_n$ is rationally triangulable in either of the following cases:
\begin{enumerate}[\hskip 4.2mm\rm(i)]
\item ${\rm tr\,deg}_kk(\bAn)^G\leqslant 1;$
\item ${\rm tr\,deg}_kk(\bAn)^G=2$ and ${\rm char}\,k=0;$
\item $G\subset {\rm Aut}\,\bAn$ and
\begin{enumerate}[\hskip -.9mm \rm (a)]
\item $\dim G\cdot x\geqslant n-1$ for some point $x\in \bAn$, or
\item ${\rm char}\,k=0$ and $\dim G\cdot x=n-2$ for some point $x\in \bAn$.
\end{enumerate}
\end{enumerate}
 \end{corollary}

\begin{corollary}[{{\rm $3$-dimensional affine space}}]\label{cor2}
If ${\rm char}\,k=0$, then every connected solvable affine algebraic subgroup of  ${\rm Aut}\,\Ab3$
is rationally triangulable.
\end{corollary}


Corollary \ref{cor2} generalizes
  \cite[Cor.\,3.2]{DF91}, where the claim is proved for one-dimensional unipotent algebraic subgroups of
  ${\rm Aut}\,\Ab3$ and ${\rm char}\,k=0$.

\begin{corollary}[{{\rm Tori}}]\label{corto}
 The following properties of an affine algebraic torus $T$ in the Cremona group $\Cn$  are equivalent:
\begin{enumerate}[\hskip 4.2mm\rm(i)]
\item $T$ is rationally triangulable,
\item $T$ is linearizable, i.e., conjugate in $\,{\mathcal C}_n$ to a subgroup of $\,{\rm GL}_n$.
\item $T$ is conjugate in $\,{\mathcal C}_n$ to the diagonal torus of $\,{\rm GL}_d$, where $d=\dim T$.
\end{enumerate}
\end{corollary}

Next we show that the nontriangulable connected solvable affine algebraic subgroups of ${\mathcal C}_n$
do exist.\;In particular, the following theorem implies that
in case (ii) of Corollary \ref{cor1} it is not possible to replace $2$ by a bigger integer.

\begin{theorem}[{{Nontriangulable subgroups}}]\label{Ns} Let $n$ be an integer $\geqslant 5$ and let ${\rm char}\,k\neq 2$.\;Every $(n-3)$-dimensional connected solvable
affine algebraic group $G$ is isomorphic to a rationally nontriangulable algebraic subgroup of the Cremona group $\Cn$.
\end{theorem}

As far as we do not claim that the subgroup in the formulation
of Theorem \ref{Ns} lies in ${\rm Aut}_k\,k[\bAn]$,
this theorem does not furnish the negative answer to Bass' Triangulability Problem.\;However, its proof  demonstrates the intimate interrelation between triangulability and Zariski's Cancellation Problem:
it shows that if there is a {\it nonrational} variety $Z$ such that $\bAn$ is isomorphic to ${\mathbf A}\!^{s}\times Z$, then the answer to Bass' Triangulability Problem is negative
(in view of this it is worth to recall that in positive characteristic Zariski's Cancellation Problem
is solved in the negative in \cite{Gu}); by Theorems \ref{birspl},
\ref{struct}, described in Section \ref{constr}, the converse it true at the  birational level.

On the other hand, in the stable range we do answer Bass' Triangulability
Problem is the affirmative.\;Namely, the following theorem shows that despite the existence
of rationally nontriangulable connected solvable affine algebraic subgroups of $\,{\mathcal C}_n$, every such subgroup is
stably rationally triangulable.\;More precisely, the following holds true.

\begin{theorem}[{{Stable triangulability}}]\label{st} Every connected solvable affine algebraic subgroup $G$ of the Cremona group $\Cn$
is rationally triangulable  in the Cremona group ${\mathcal C}_m$ for every
\begin{equation}\label{ineq}
m\geqslant 2n-{\rm tr\,deg}_k k(\bAn)^G.
\end{equation}
\end{theorem}

Theorem \ref{st} generalizes
  \cite[Thm.\,3.1]{DF91}, where the statement is proved
  for one-dimensional unipotent algebraic subgroups of  ${\rm Aut}_k \bAn$, assuming ${\rm char}\,k = 0$.

Next we obtain a general construction of all rationally triangulable subgroups of $\,{\mathcal C}_n$, see
Theorem \ref{struct} in Section 3.\;As an application, it leads to the
 classification (given below in Corollary \ref{1sbgr}) of rationally triangulable  one-dimensional  connected uni\-potent algebraic subgroups in $\,{\mathcal C}_n$ up to conjugacy.\;In this classification we use the following terminology.\;

 A one-dimen\-si\-o\-nal
 connected unipotent algebraic subgroup $G$ of $\,{\mathcal C}_n$, identified with $k^+$ by means of an isomorphism $G\to k^+$,
is called {\it standard} if $x_1,\ldots, x_{n-1}\in k[\bAn]^G$ and, for every $t\in k^+$, the following holds:
\begin{enumerate}[\hskip 2mm\rm(i)]
\item for ${\rm char}\,k=0$, we have
$$t(x_n)=x_n+t$$ (in this case $G$ is also called the {\it translation});
    \item for ${\rm char}\,k=p>0$, we have
\begin{equation*}
t(x_n)=f_n+c_1t^{p^{i_1}}+\cdots +
    c_dt^{p^{i_d}},
\end{equation*}
where  all $c_j$'s are the nonzero elements of
$k(x_1,\ldots, x_{n-1})$,
and
$i_1<\cdots<i_d$ are the nonnegative integers.
\end{enumerate}

\begin{corollary}[\hskip -.8mm One-dimensional rationally triangulable unipotent subgroups]\label{1sbgr}
A one-dimen\-si\-o\-nal  connected unipotent algebraic subgroup of $\,{\mathcal C}_n$ is rationally triangulable if and only if it is conjugate in $\,{\mathcal C}_n$ to a standard subgroup.
\end{corollary}

Corollary \ref{1sbgr} generalizes
 \cite[Thm.\,2.2]{DF91}, where the claim is proved for ${\rm char}\,k=0$.



\begin{corollary}[{{\rm Low-dimensional affine spaces}}]\label{23} Let $U$ be a one-dimen\-si\-o\-nal  connected unipotent algebraic subgroup of $\,{\mathcal C}_n$.
\begin{enumerate}[\hskip 4.2mm\rm(i)]
 \item If $n=2$, then $U$ is conjugate in $\,{\mathcal C}_2$ to a standard subgroup.
\item If ${\rm char}\, k=0$ and $n=3$, then $U$ is  conjugate in $\,{\mathcal C}_3$ to the translation.
\end{enumerate}
\end{corollary}


By \cite{Re}, if ${\rm char}\,k=0$, $n=2$, and $U\subset {\rm Aut}\,{\mathbf A}\!^2$, then ``in $\,{\mathcal C}_2$'' in Corollary \ref{23}(i)
may be replaced by  ``in ${\rm Aut}\,{\mathbf A}\!^2$''.

By \cite{Ka2},  for $k=\mathbf C$, $n=3$, $U\subset {\rm Aut}\,{\mathbf A}\!^3$,
if $U$ acts on  ${\mathbf A}\!^3$
freely, then $U$ is conjugate in ${\rm Aut}\,{\mathbf A}\!^3$ to the translation.\;Corollary \ref{23}(ii) shows
that, allowing conjugation in ${\mathcal C}_3$, the ``if'' assumption in this result may be dropped, i.e., conjugacy to the translastion becomes true for every $U$ in ${\mathcal C}_3$.

The proofs of Theorems \ref{ext1}--\ref{st}
and Corollaries \ref{cor1}--\ref{corto}
are given in Sec\-tion\;2. Theorem \ref{struct} is formulated and proved, together with Corollaries \ref{1sbgr}, \ref{23},  in Section 3.


\vskip 2mm
\noindent{\it Notation and conventions.} We use freely the standard notation and conventions of
\cite{Bor}, \cite{Spr}, \cite{PV94} and  refer to \cite{Ro56}, \cite{Ro61}, \cite{Ro63},
\cite{PV94}, \cite{P13} regarding the definitions and basic properties of rational and regular (morphic) actions of algebraic groups.

 Unless otherwise specified, we will assume that every field
appearing below contains $k$ and every embedding of fields
restricts to the identity map on $k$.

\vskip 2mm

{\it Acknowledgements.} My thanks go to the referees for their comments.



\section{Proofs of Theorems \ref{ext1}--\ref{st}
and Corollaries \ref{cor1}, \ref{cor2}}

\begin{proof}[Proof of Theorem {\rm \ref{ext1}}] \


1. The assumptions on $Q$ and $P$ imply that there is an
irreducible variety
$X$ such that for
\begin{equation}\label{Y}
\begin{gathered}
Y:=\bA1\times X,\\[-.7mm]
\rho\colon Y\to X,\;\;(a, x)\mapsto x,
\end{gathered}
\end{equation}
we may (and shall) identify $P$ with $k(Y)$:
\begin{equation}\label{P}
P=k(Y),
\end{equation}
and $Q$ with the image of embedding $\rho^*\colon k(X)\hookrightarrow k(Y)$:
\begin{equation}\label{Q}
Q=\rho^*(k(X)).
\end{equation}

The actions of $G$ on $Q$ and $P$ determine the rational actions of $G$ on $X$ and $Y$.\;The  action on $Y$ is faithful and the morphism $\rho$ is equivariant with respect to these actions.

By Weil's regularization theorem \cite{We55} (see also \cite[Thm.\,1]{Ro56}) there are
\begin{enumerate}[\hskip 4.2mm\rm ---]
\item irreducible algebraic varieties $\widetilde X$ and $\widetilde Y$;
\item nonempty open subsets $X_0$ and $Y_0$ in, respectively, $X$ and $Y$;
\item open embeddings $
\iota_1\colon X_0\hookrightarrow \widetilde X$, $
\iota_2\colon Y_0\hookrightarrow \widetilde Y$
\end{enumerate}
such that the induced rational actions of $G$ on $\widetilde X$ and $\widetilde Y$ are regular (morphic).\;The action of $G$ on $\widetilde Y$ is faithful since that on $Y$ is.\;We identify $X_0$ and $Y_0$ with the images of, respectively,  $\iota_1$ and $\iota_2$, that yields the natural identifications
\begin{equation}\label{tilde}
k(X)=k(\widetilde X),\quad k(Y)=k(\widetilde Y).
\end{equation}

By Rosenlicht's theorem on generic quotients \cite[Thm.\,2]{Ro56}, replacing $\widetilde X$, $X_0$, $\widetilde Y$, and $Y$ by the appropriate open subsets, we may (and shall) assume that for the actions of $G$ on $\widetilde X$ and $\widetilde Y$ respectively there are the geometric quotients
\begin{equation*}
\pi^{\ }_{{\widetilde X}, G}\colon {\widetilde X} \to {\widetilde X}/G,\qquad \pi^{\ }_{{\widetilde Y}, G}\colon {\widetilde Y}\to {\widetilde Y}/G.
\end{equation*}
In particular, $\pi^{\ }_{{\widetilde X}, G}$ and $\pi^{\ }_{{\widetilde Y}, G}$  are equidimensional morphisms, their fibers are
$G$-orbits, and, in view of \eqref{P}, \eqref{Q}, \eqref{tilde},
\begin{equation}\label{isof}
\rho^*\circ\pi^{*}_{{\widetilde X}, G}\colon k({\widetilde X}/G)\xrightarrow{\simeq} Q^G,\qquad
\pi^{*}_{{\widetilde Y}, G}\colon k({\widetilde Y}/G)\xrightarrow{\simeq} P^G.
\end{equation}
 Since $\dim G=1$ and $G$ acts on $\widetilde Y$ faithfully, every fiber of $\pi^{\ }_{{\widetilde Y}, G}$
is one-di\-men\-sional; in view of \eqref{Y}, this yields
\begin{equation}\label{dimwY/G}
\dim{\widetilde Y}/G=\dim Y-1=\dim X.
\end{equation}

The morphism $\rho$ induces a $G$-equivariant dominant rational map
$$\widetilde \rho\colon \widetilde Y\dashrightarrow \widetilde X.$$
Since its domain of definition
is $G$-stable, replacing the varieties again by
 the appropriate open subsets
we may (and shall) 
assume that $\widetilde \rho$ is a $G$-equivariant surjective morphism.

Thus we obtain the following
commutative
diagram
\begin{equation}\label{dia}
\begin{matrix}
\xymatrix@C=8mm{
Y\ar[d]_{\rho}&\;Y_0\,\ar@{_{(}->}_{\hskip 1mm\iota_4}[l]\ar@{^{(}->}^{\hskip 2mm\iota_2}[r]\ar[d]^{\rho_0}&
\widetilde Y\ar[d]^{\widetilde \rho}\ar[r]^{\hskip -3mm\pi^{\ }_{{\widetilde Y}, G}}&{\widetilde Y}/G\\
X&\;X_0\,\ar@{_{(}->}_{\hskip 1mm\iota_3}[l]\ar@{^{(}->}^{\hskip 2mm\iota_1}[r]&
\widetilde X\ar[r]^{\hskip -2mm\pi^{\ }_{{\widetilde X}, G}}&{\widetilde X}/G
}
\end{matrix}\;\;\;,
\end{equation}
where $\rho_0:=\rho|_{Y_0}={\widetilde\rho}|_{Y_0}$ and $\iota_3,\iota_4$ are
the identical embeddings.\;Finally, replacing $X_0$ and $Y_0$ by the appropriate open subsets, we may (and shall) assume that
$\rho_0(Y_0)=X_0$.





Now we shall consider separately two arising possibilities.
\vskip 1mm

2. First, consider the case
\begin{equation}\label{QSQ}
Q^G=Q.
\end{equation}
By \eqref{Q}, condition \eqref{QSQ}
is equivalent to triviality of the action of $G$ on $\widetilde X$.

From $Q\subset P$ and \eqref{QSQ} we obtain
\begin{equation}\label{QPSP}
Q\subseteq P^G\subset P,
\end{equation}
From \eqref{P}, \eqref{tilde},  \eqref{dimwY/G} we infer ${\rm tr\,deg}_{P^G}P=1$,
and \eqref{Y}, \eqref{P}, \eqref{tilde} yield
${\rm tr\,deg}_{Q}P=1$.
Whence by
\eqref{QPSP} we obtain
${\rm tr\,deg}_{Q}P^G=0$.\;Since $P$ is a $1$-extension of $Q$,
 by L\"uroth's theorem
 (\cite{Lu}, see also, e.g.,  \cite[\S73]{vdW})
  the latter equality implies that $P^G=Q$.\;Thus, by \eqref{QSQ}, in the case under consideration
  we have $P^G=Q^G$.\;This proves claim (i) of Theorem \ref{ext1}.

\vskip 1mm

3. Now consider the
case
$Q^G\varsubsetneq Q$, i.e.,
$G$ acts on $\widetilde X$ nontrivially.
Every fiber of $\pi^{\ }_{{\widetilde X}, G}$ is then a one-dimensional
$G$-orbit; whence
\begin{equation}\label{dimwXS}
\dim {\widetilde X}/G=\dim X-1.
\end{equation}
In view of \eqref{P}, \eqref{Q}, \eqref{tilde},  \eqref{dimwY/G}, \eqref{dimwXS},
we have
\begin{equation}\label{trtr}
{\rm tr\,deg}_{Q^G}P^G=1.
\end{equation}

 Since $G$ is a connected solvable affine algebraic group,  by Rosenlicht's cross-section theorem \cite[Thm.\,10]{Ro56} there is a rational section
 \begin{equation*}
 \sigma\colon {\widetilde X}/G \dashrightarrow \widetilde X
 \end{equation*}
of $\pi^{\ }_{\widetilde X, G}$, i.e. a rational map such that
\begin{equation}\label{id}
\pi^{\ }_{\widetilde X, G}\circ\sigma={\rm id}_{\widetilde X}.
\end{equation}
Since $g\circ\sigma$ for every element $g\in G$ is also a rational section of $\pi^{\ }_{\widetilde X, G}$, we may (and shall) assume that
$\sigma({\widetilde X}/G) \cap X_0\neq \varnothing$.\;This implies that there is a locally closed irreducible subvariety $S\subset X_0$ such that
\begin{equation}
\label{oe}
\mbox{$\pi^{\ }_{\widetilde X, G}|_S^{\ }\colon S\to {\widetilde X}/G$\;
is an open embedding.}
\end{equation}
 In view of \eqref{isof}, this means that
\begin{equation}\label{isorest}
\mbox{$k(\widetilde X)^G\to k(S)$,\;\;$f\mapsto f|_S,$ is a well-defined isomorphsim.}
\end{equation}
In particular, in view of \eqref{dimwXS}, we have
\begin{equation}\label{===}
\dim S=\dim X-1.
\end{equation}

From
\eqref{Y} we obtain that
\begin{equation}\label{rS}
\rho^{-1}(S)=\bA1\times S.
\end{equation}
Consider in $\widetilde Y$ the locally closed irreducible subvariety
\begin{equation}\label{Z}
Z:=\rho^{-1}(S)\cap Y_0.
\end{equation}
From \eqref{rS} and \eqref{Z} we infer that
\begin{gather}
\mbox{$k(Z)$ is a $1$-extension of $K:=\rho|_Z^*(k(S))$,}\label{1ext}
\end{gather}
and from \eqref{dimwY/G}, \eqref{===} that
\begin{gather}
\dim Z=\dim X=\dim {\widetilde Y}/G.\label{dZ}
\end{gather}

We claim that the morphism
\begin{equation}\label{domi}
\zeta:=\pi^{\ }_{{\widetilde Y}, G}|^{\ }_Z\colon Z\to {\widetilde Y}/G
\end{equation}
is dominant.\;In view of \eqref{dZ}, to prove this, it suffices to show that,
for every point $z\in Z$, the fiber $\zeta^{-1}(\zeta(z))$ is finite.\;Assume the contrary.\;Since $\zeta^{-1}(\zeta(z))=Z\cap \mathscr O$, where $\mathscr O:=\pi^{-1}_{{\widetilde Y}, G}(\pi^{\ }_{{\widetilde Y}, G}(z))$ is a certain $G$-orbit,
we then infer
from $\dim \mathscr O=1$
that
\begin{equation}\label{ZO}
\dim (Z\cap \mathscr O)=1.
\end{equation}
Since
$\widetilde \rho$
is a $G$-equivariant morphism,
${\mathscr O}':={\widetilde \rho}(\mathscr O)$ is also
 a $G$-orbit.\;Hence $\dim {\mathscr O}'=1$, because all  $G$-orbits in $\widetilde X$ are one-dimensional.\;The latter equa\-li\-ty and \eqref{ZO} imply that
 ${\widetilde \rho}(Z\cap \mathscr O)$ is an open subset of  ${\mathscr O}'$; in particular,
it is infinite.\;On the other hand,  \eqref{Z} yields that ${\widetilde \rho}(Z\cap \mathscr O)$ lies in $S\cap {\mathscr O}'$. Since, by \eqref{oe},
 the latter is a single point, we obtain a contradiction.
 This
 proves the claim.

Thus, since $\zeta$ is dominant, it defines an embedding of fields
$\zeta^*\colon k(\widetilde Y/G)\hookrightarrow k(Z)$. In view of
\eqref{P}, \eqref{tilde}, this means that
\begin{equation}\label{isorest2}
\mbox{$\tau\colon P^G\hookrightarrow k(Z)$,\;\;$f\mapsto f|_Z,$ is a well-defined embedding,}
\end{equation}
and \eqref{Q}, \eqref{tilde}, \eqref{isorest}, \eqref{1ext} imply that
\begin{equation}\label{isorest3}
\mbox{$\tau\colon Q^G\hookrightarrow K$ is an isomorphism.}
\end{equation}

Thus \eqref{isorest2}, \eqref{isorest3} yield the following commutative diagram
\begin{equation*}
\begin{matrix}
\xymatrix{
k(Z)& \tau(P^G)\ar@{_{(}->}[l]_{\rm id}&K\ar@{_{(}->}[l]_{\rm id}\\
&P^G\ar[u]^{\tau}_{\simeq}&Q^G\ar@{_{(}->}[l]_{\rm id}\ar[u]^{\tau}_{\simeq}
}
\end{matrix}\;\;.
\end{equation*}
In view of \eqref{trtr}, \eqref{1ext}, this diagram and L\"uroth's theorem imply that
$P^G$ is a $1$-extension of $Q^G$.\;This completes the proof of claim (ii) of Theorem \ref{ext1}.
\end{proof}

\begin{proof}[Proof of Theorem {\rm \ref{ext1g}}]\

We argue by induction on $\dim G$.\;The statement being clear for trivial $G$, we need, assuming $\dim G>0$, to prove the inductive step.

Since $G$ is a connected solvable affine algebraic group, it contains a closed connected normal subgroup $N$
such that $\dim G/N=1$, see, e.g., \cite[p.\,6--03, Cor.\,1]{Gro}.\;The group $G/N$ naturally acts on $P^N$ and $Q^N$, and we have
\begin{equation}\label{S/N}
(P^N)^{G/N}=P^G\;\;\mbox{and}\;\;(Q^N)^{G/N}=Q^G.
\end{equation}
By the inductive assumption one of the following holds:
\begin{enumerate}[\hskip 4.2mm\rm(a)]
\item $P^N=Q^N$;
\item $P^N$ is a $1$-extension of $Q^N$.
\end{enumerate}

 If (a) is fulfilled, then $P^G=Q^G$ in view of \eqref{S/N}, i.e., (i) holds.

 If (b) is fulfilled, then, since $G/N$ is a one-dimensional connected solvable affine algebraic group,
 \eqref{S/N} and Theorem \ref{ext1} imply that either (i) or (ii) holds. 
This completes the proof. 
\end{proof}


\begin{proof}[Proof of Theorem {\rm \ref{cr}}]\

(i)$\Rightarrow$(ii) Assume that $k(\bAn)^G$ is a purely transcendental field extension of $k$.
Then there is a flag
\begin{equation}\label{I}
k(\bAn)^G=:I_t\supset I_{t-1}\supset\cdots\supset I_1\supset I_0:=k
\end{equation}
of the subfields of $k(\bAn)^G$ such that $I_i/I_{i-1}$ is a
$1$-extension for every $i>0$.

Since $G$ is a connected solvable affine algebraic group, there is a flag of its closed connected normal subgroups
\begin{equation}\label{sbgs}
G=:G_0\supset G_1\supset\cdots\supset G_{s-1}\supset G_s:=\{e\}
\end{equation}
such that $\dim G_{i-1}/G_{i}=1$ for every $i>0$, see \cite[p.\,6--03, Cor.\,1]{Gro}.
From \eqref{sbgs} we obtain the following flag of $G$-stable subfields of $k(\bAn)$:
\begin{equation}\label{KKK}
k(\bAn)=K_s\supseteq K_{s-1}\supseteq\cdots \supseteq K_1\supseteq K_0=k(\bAn)^G,\quad\mbox{$K_i:=k(\bAn)^{G_i}$}.
\end{equation}

By construction,
\begin{equation}\label{Ki}
K_{i-1}=k(\bAn)^{G_{i-1}}=(k(\bAn)^{G_i})^{G_{i-1}/G_i}=K_{i}^{G_{i-1}/G_i}\quad \mbox{for every $i>0$}.
 \end{equation}
 If the action of $G_{i-1}/G_i$ on $K_{i}$ is trivial, then \eqref{Ki} yields $K_i=K_{i-1}$.\;If it is nontri\-vi\-al, then, since $G_{i-1}/G_i$ is a one-dimensional connected solvable  affine algebraic group, we infer from \eqref{Ki} and  \cite[Thm.\;1]{P15}  (or \cite[Thm.\;1]{Ma63} if ${\rm char}\,k=0$)
that $K_i/K_{i-1}$ is a
$1$-extension.\;Therefore, once repetitions
in flag \eqref{KKK} are eliminated, we obtain a flag
\begin{equation}\label{L}
k(\bAn)=:L_d\supset L_{d-1}\supset\cdots\supset L_1\supset L_0:=k(\bAn)^G
\end{equation}
of $G$-stable subfields of $k(\bAn)$ such that $L_i/L_{i-1}$ is
a $1$-extension for every  $i>0$. From \eqref{L} and \eqref{I} we then obtain the flag
\begin{equation*}
k(\bAn)=:L_d\supset L_{d-1}\supset\cdots\supset L_1\supset L_0\!=\!I_t
\supset I_{t-1}\supset\cdots\supset I_1\supset I_0:=k
\end{equation*}
of $G$-stable subfields of $k(\bAn)$ whose ``levels'' are
$1$-extensions.\;By Definition \ref{triang} the group $G$ is then rationally triangulable.

\vskip 2mm

(ii)$\Rightarrow$(i) Conversely, assume that  the group $G$ is rationally triangulable and consider
a flag \eqref{K} of $G$-stable subfields of $k(\bAn)$.\;Passing to the $G$-invariant subfields, we then obtain the following flag of subfields of
$k(\bAn)^G$:
\begin{equation}\label{KG}
k(\bAn)^G=K_n^G\supseteq K_{n-1}^G\supseteq \cdots\supseteq K_1^G\supseteq K_0^G=k.
\end{equation}
By Corollary \ref{ext1g}, for every $i=1,\ldots, n$, the field
$K^G_i$ is a purely transcendental extension of $K_{i-1}^G$.\;Therefore,
\eqref{KG} yields that $k(\bAn)^G$ is a purely transcendental extension of $k$.\;This completes the proof.
\end{proof}



\begin{proof}[Proof of Corollary {\rm \ref{cor1}}]
In cases (i) and (ii), the claim follows, in view of Theorem \ref{cr},  from, respectively, L\"uroth's theorem and Castelnuovo's theorem (see, e.g., \cite[Chap.\,V, 6.2.1]{Ha}).\;In case (iii), it follows from (i) and (ii), since, by \cite[Thm.\,2]{Ro56},
${\rm tr\,deg}_kk(\bAn)^G=n-\max_{a\in \bAn}G\cdot a$.
\end{proof}



\begin{proof}[Proof of Corollary {\rm \ref{cor2}}] This follows from Corollary \ref{cor1}.
\end{proof}

\begin{proof}[Proof of Corollary {\rm \ref{corto}}] In view of Theorem \ref{cr}, the equivalence (i)$\Leftrightarrow$(ii) follows from Theorem 2.4 of \cite{P13}
(the assumption ${\rm char}\,k=0$ made in \cite{P13} is not used in the proof
of this theorem).\;The equivalence (ii)$\Leftrightarrow$(iii) follows from Corollary 4 of \cite{P13-2} (the assumption ${\rm char}\,k=0$ made in \cite{P13-2} is not used in the proof
of this corollary).
\end{proof}



\begin{proof}[Proof of Theorem {\rm \ref{Ns}}] By \cite{S-B04} (where it is assumed that ${\rm char}\,k\neq 2$)
there exists a nonrational threefold $X$ such that ${\mathbf A}\!^2
\times X$ is birationally isomorphic to ${\mathbf A}\!^5$.\;Hence we may (and shall)
\begin{equation}\label{biso}
\mbox{
 fix a birational isomorphism ${\mathbf A}\!^{n-3}\times X\dashrightarrow \bAn$.}
\end{equation}

Since the underlying variety of $G$ is rational (see \cite[p.\,5-02, Cor.]{Gr}), we also may (and shall)
fix a birational isomorphism between it and ${\mathbf A}\!^{n-3}$. We then obtain from the action of $G$ on itself by left translations a faithful rational action of $G$ on ${\mathbf A}\!^{n-3}$ such that
 \begin{equation}\label{k}
 k({\mathbf A}\!^{n-3})^G=k.
 \end{equation}
In turn,  the latter action yields a faithful rational action of $G$
on ${\mathbf A}\!^{n-3}\times X$ via the first factor.\;By
\cite[Cor.\,of\;Prop.\,II.6.6]{Bor} and \eqref{k}, for this action,
\begin{equation}\label{inva}
\mbox{$k({\mathbf A}\!^{n-3}\times X)^G$ and $k(X)$ are isomorphic.}
\end{equation}

Thus, given \eqref{biso}, we obtain a faithful rational action of $G$ on $\bAn$ such that
the field $k(\bAn)^G$ is  isomorphic to $k(X)$, and hence is not purely transcendental over $k$ according to the choice of $X$.

By
Theorem \ref{cr} we then conclude that the algebraic subgroup of $\mathcal C_n$ determined by this action is isomorphic to $G$ and rationally nontriangulable.
\end{proof}



\begin{proof}[Proof of Theorem {\rm \ref{st}}] First, we shall introduce some notation.\;Given a rational action of an algebraic group $H$ on an irreducible
algebraic variety $Z$, we denote by
$Z\dss H$ a rational quotient of this action, i.e.,
an irreducible variety (uniquely defined up to birational isomorphism) such that there exists an isomorphism
$k(Z\dss H)\to k(Z)^H$, restricting to the identity map on $k$.\;We shall write $X\approx Y$ if
$X$ and $Y$ are birationally isomorphic irreducible varieties.


By  \cite[Thm.\;1]{P15}  (or \cite[Thm.\;1]{Ma63} if ${\rm char}\,k=0$) we have
\begin{equation}\label{x}
\mbox{$\bAn\approx \bAn\dss G\times \bAs$,\;\;where $s:=n-{\rm tr\,deg}_kk(\bAn)^G$.}
\end{equation}
 Take an integer $m$ such that \eqref{ineq} holds and consider $G$ as a subgroup of $\;\mathcal C_m$  via embedding \eqref{embC}. The arising rational action of $G$ on
  $\bAm=\bAn\times {\mathbf A}\!^{m-n}$ is that through the first factor.\;Therefore
\begin{equation}\label{simsim}
\left.\begin{split}
\bAm\dss G&=(\bAn\dss G)\times {\mathbf A}\!^{m-n}\\
&\overset{\eqref{ineq}}{=
}(\bAn\dss G)\times\bAs\times {\mathbf A}\!^{m-n-s}\\
&\overset{\eqref{x}}{\approx} \bAn\times {\mathbf A}\!^{m-n-s}={\mathbf A}\!^{m-s}
\end{split}\right\}\;\;.
\end{equation}
By \eqref{simsim} the field $k(\bAm)^G$ is purely transcendental over $k$. Hence  by Theorem \ref{cr} the group $G$ is rationally triangulable in ${\mathcal C}_m$.
\end{proof}


\section{Construction
$\circledast$}\label{constr}

Now we shall  give
a general construction of all rationally triangulable subgroups of $\,{\mathcal C}_n$.\;It is prompted by the following result from
 \cite{P15}
 that yields a general
 construction of all connected solvable affine algebraic subgroups of $\,{\mathcal C}_n$:


 \begin{theorem}[{{\rm Standard invariant open subsets \cite[Thm.\,3]{P15}}}]\label{birspl} Let $X$ be an irreducible variety endowed with a re\-gu\-lar action of a connected solvable  affine algebraic group $S$.\;Then for the restriction of this action
on a certain $S$-stable nonempty open subset $\,Q$ of $X$ there exist
\begin{enumerate}[\hskip 2.2mm ---]
\item the geometric quotient $\pi^{\ }_{S, Q}\colon Q\to Q/S$;
\item an isomorphism  $\varphi\colon Q\to {\bf A}^{\hskip -.5mm r, u}\times (Q/S)$, where
$${\bf A}^{\hskip -.5mm r, u}:=\{(\alpha_1,\ldots, \alpha_{r+u})\in {\bf A}^{\hskip -.5mm r+u}\mid \alpha_i\neq 0 \;\mbox{for every $i\leqslant r$}\},\;\;r\geqslant 0,\;u\geqslant 0,$$
\end{enumerate}
such that the natural projection  ${\bf A}^{\hskip -.5mm r, u}\times (Q/S)\to Q/S$ is the geometric quotient of the regular action of $\;S$ on ${\bf A}^{\hskip -.5mm r, u}\times (Q/S)$ induced by $\varphi$.
\end{theorem}

Theorem \ref{birspl} leads to the following

\vskip 1mm
\noindent{\it Construction 
$\circledast$}

\vskip 1mm

Let $S$ be a connected solvable affine algebraic group and let $Z$ be an irreducible variety such that,
for some nonnegative integers $r$, $u$,
 \begin{equation}\label{rrrat}
 \mbox{the variety $Z\times {\bf A}^{\hskip -.5mm r, u}$ is rational.}
 \end{equation}
Consider a map
 \begin{equation*}
 \varphi\colon S\times Z\to{\rm Aut}\,  {\bf A}^{\hskip -.5mm r, u}
 \end{equation*}
that has the following properties:
 \begin{enumerate}[\hskip 2mm\rm(i)]
 \item $\varphi$ is an algebraic family \cite{R64}, \cite{P14}, i.e.,
 \begin{equation*}
 S\times Z\times {\bf A}^{\hskip -.5mm r, u}\to {\bf A}^{\hskip -.5mm r, u},\quad (s, z, a)\mapsto \varphi(s, z)(a),
 \end{equation*}
 is a morphism;
 \item for every point $z\in Z$, the algebraic family
 \begin{equation*}
 \varphi_z\colon
 S\to {\rm Aut}\,  {\bf A}^{\hskip -.5mm r, u},\quad s\mapsto \varphi(s, z),
 \end{equation*}
 is a homomorphism determining a transitive action of
 $S$ on ${\bf A}^{\hskip -.5mm r, u}$.
 \end{enumerate}

These data determine
the regular action of $S$ on $Z\times {\bf A}^{\hskip -.5mm r, u}$ by the formula
\begin{equation}\label{actio}
S\times (Z\times {\bf A}^{\hskip -.5mm r, u})\to Z\times {\bf A}^{\hskip -.5mm r, u}, \quad
\big(s, (z, a)\big)\mapsto \big(z, \varphi(s, z)(a)\big).
\end{equation}
By (ii),  the orbits of this action are  the fibers of the projection
\begin{equation*}\label{quoti}
\pi\colon Z\times {\bf A}^{\hskip -.5mm r, u}\to Z,\quad (z, a)\mapsto z,
\end{equation*}
and, more precisely, for every point $z\in Z$, the variety ${\bf A}^{\hskip -.5mm r, u}$ endowed with the $S$-action determined by $\varphi_z$ is $S$-isomorphic to
the  fiber $\pi^{-1}(z)$.\;By \cite[Prop.\,II.6.6]{Bor} this implies that
$k(Z\times {\bf A}^{\hskip -.5mm r, u})^S\simeq k(Z)$.

In
view of \eqref{rrrat}, for $n=\dim Z+r+u$,
 we may (and shall)
 fix a birational isomorphism
 \begin{equation*}
 \gamma\colon Z\times {\bf A}^{\hskip -.5mm r, u}\dashrightarrow \bAn.
 \end{equation*}
 Then $\gamma$ and action  \eqref{actio} determine
 a rational action of $S$ on $\bAn$.\;The image of the homomorphism $S\to \mathcal C_n$
 determined by this rational action is a connected solvable
affine algebraic subgroup $G$ of $\mathcal C_n$, and for this $G$ we have
 \begin{equation}\label{fie}
 k(\bAn)^G \simeq k(Z).
 \end{equation}
We say that $G$ is a {\it subgroup of $\,\mathcal C_n$ obtained by Construction
$\circledast$}.

\vskip 2mm

Theorem\;\ref{birspl}
(combined with Weil's regularization theorem \cite{We55})
implies that
this construction is universal, i.e., every connected solvable
affine algebraic subgroup  of $\,\mathcal C_n$ is obtained by Construction $\circledast$.



\begin{example}
[One-dimensional connected unipotent subgroups of ${\mathcal C}_n$]\label{exa}
 {\rm Let $G$ be the one-dimensional 
additive group $k^+$.\;In
view of (ii), we then have
$r=0$, $u=1$, i.e.,
$\varphi_z\colon G=k^+\to {\rm Aut}\,{\bf A}^1$ for every $z\in Z$.\;Since ${\rm Aut}\,{\bf A}^1=T\ltimes N$, where $T$ is a one-dimensional torus and $N\simeq k^+$ is the
subgroup consisting of all translations $
{\bf A}^1\to {\bf A}^1$, $a\mapsto a+t$, $t\in k^+$,
this means that every
$\varphi_z$ may be identified with a surjective homomorphism
$k^+\to k^+$.\;What happens next depends on ${\rm char}\,k$, see \cite[Lemma\;3.3.5]{Spr}.\;

Namely, a map $k^+\!\!\to k^+$ is
a homomorphism if and only if it
has the fol\-lo\-wing form:

\hskip 1mm (i) {\rm case ${\rm char}\,k=0$}:\;\;
$t\mapsto ct,$
where $c$ is a fixed element of $k$,

(ii) {\rm case ${\rm char}\,k=p>0$}:\;\;$ t\mapsto \alpha_1t^{p^{i_1}}+\cdots + \alpha_dt^{p^{i_d}},$
 where $\alpha_1,\ldots, \alpha_d$ are the nonzero elements of $k$ and $i_1,\ldots, i_d$ is an increasing sequence of nonnegative integers.

Since every one-dimensional connected  unipotent algebraic group is isomorphic to $k^+$ (see, e.g.,\;\cite[Thm.\;3.4.9]{Spr}),
this yields the following  ge\-ne\-ral method of constructing
connected one-dimensional unipotent algebraic subgroups of $\,{\mathcal C}_n$.

Take
an irreducible variety $Z$ such that $Z\times {\mathbf A}^1$ and $\bAn$ are birationally isomorphic.\;If ${\rm char}\,k=0$, fix
 a nonzero regular function $f\in k[Z]$.\;If  ${\rm char}\,k=p>0$, fix a sequence of nonnegative integers
 $i_1<\ldots < i_d$ and a sequence of nonzero regular functions $f_1,\ldots,f_d\in k[Z]$.\;Consider
  the action of $S=k^+$ on $Z\times {\mathbf A}^1$ defined by the formula
 \begin{equation}\label{acttt}
 \begin{split}
 S\!\times\! (Z\!\times\! {\mathbf A}^1)&\to Z\!\times\! {\mathbf A}^1,\\ (s, (z, a))&\!\mapsto\!\begin{cases}\hskip -.5mm
 (z, a\!+\!f(z)s)&\mbox{if ${\rm char}\,k\!=\!0$},\\
 \hskip -.5mm (z, a\!+\!f_1(z)s^{p^{i_1}}\!+\!\cdots \!+\! f_d(z)s^{p^{i_d}})&\mbox{if ${\rm char}\,k\!=\!p\!>\!0$}.
 \end{cases}
 \end{split}
 \end{equation}
 Then  action \eqref{acttt} and a fixed birational isomorphism $\gamma\colon Z\times {\mathbf A}^1 \dashrightarrow\bAn$ determine a one-dimensional connected  unipotent algebraic subgroup $G$ of $\,{\mathcal C}_n$,
 and every such subgroup is obtained this way.
}
\end{example}

\begin{theorem}[{{\rm Structure theorem}}]\label{struct} The following properties of a connected solvable affine algebraic subgroup $G$ of the Cremona group
 $\,\mathcal C_n$ are equivalent:
 \begin{enumerate}[\hskip 2.2mm\rm(i)]
 \item $G$ is rationally triangulable;
 \item $G$ is obtained by Construction $\circledast$, in which the variety $Z$ is rational.
 \end{enumerate}
\end{theorem}
\begin{proof}
This follows from \eqref{fie} and Theorem \ref{cr}.
\end{proof}



\begin{proof}[Proof of Corollary {\rm \ref{1sbgr}}] Let $G$ be a rationally triangulable one-dimen\-si\-o\-nal connected unipotent  algebraic subgroup of $\,{\mathcal C}_n$.\;By Theorem \ref{struct} and Example \ref{exa}, $G$ is obtained from the action of $S=k^+$ on $Z\times  {\mathbf A}^1$ defined by formulas \eqref{acttt}, where $Z$ is  a rational variety.\;Therefore there are functions $f_1,\ldots, f_n\in k(\bAn)$ such that $k(f_1,\ldots, f_n)=k(\bAn)$, $f_1,\ldots, f_{n-1}\in k(\bAn)^G$, and $t(f_n)$ for every element $t\in S$ is the following function:
\begin{enumerate}[\hskip 4.2mm\rm(i)]
\item if ${\rm char}\,k=0$, then $t(f_n)=f_n+ct$, where $c\in k(f_1,\ldots, f_{n-1})$,
\item if ${\rm char}\,k=p>0$, then
$$
t(f_n)=f_n+b_1t^{p^{i_1}}+\cdots +b_dt^{p^{i_d}},
$$
where $b_j\in k(f_1,\ldots, f_{n-1})$, $b_j\neq 0$ for every $j$ and $i_1<\cdots <i_d$ are the nonzero integers.
\end{enumerate}

In case (i), replacing $f_n$ by
$f_n/c$ we may (and shall) assume that $c\!=\!1$. Then, conjugating $S$ by means of $\varphi\in {\mathcal C}_n$ such that
$\varphi(f_i)=x_i$ for every $i$ (see \eqref{xi}), we obtain a standard subgroup.

The converse (that standard subgroups are rationally triangulable) is clear.
\end{proof}

\begin{proof}[Proof of Corollary {\rm \ref{23}}] This follows from Corollaries \ref{cor1}, \ref{cor2}, and \ref{1sbgr}.
\end{proof}


\begin{thebibliography}{vdW\,1971}


\bibitem[Ba\,1984]{B84} H.\;Bass, {\it A non-triangular action of $\;{\mathbb G}_a$ on ${\mathbb A}^3$},
J. Pure Appl. Algebra {\bf 33} (1984), 1--5.

\bibitem[Bo\,1991]{Bor} {\rm  A.\;Borel}, \emph{Linear Algebraic
Groups},
 2nd enlarged ed., Graduate Texts in Mathematics,
  Vol. 126, Sprin\-ger-Verlag, 1991.

\bibitem[DF\,1991]{DF91} J.\;K.\;Deveney, D.\;R.\;Finston, {\it Rationally triangulable automorphisms},
J. Pure Appl. Algebra {\bf 72} (1991), 1--4.

\bibitem[Gr\,1956]{Gro} A.\;Grothendieck, {\it Les th\'eor\`emes de structure fondamentaux pour les grou\-pes alg\'ebriques affines}, in:\;{\it S\'eminaire C. Chevalley, $1956$--$1958$, Classification des Groupes de Lie Alg\'ebriques}, Vol. 1, E.N.S, Paris, 1958, Expos\'e n${}^\circ$\,6, 10.12.1956.

         \bibitem[Gr\,1958]{Gr}
{A.\;Grothendieck}, {\it Torsion homologique et sections
rationnelles}, in: \emph{S\'eminaire C. Chevalley, $1958$, Anneaux de Chow et Applications},
E.N.S, Paris, 1958, Expos\'e n${}^\circ$\,5, 16.06.1958.


\bibitem[Gu\,2014]{Gu} N.\;Gupta, {\it On Zariski's Cancellation Problem in positive characteristic}, Adv. in Math. {\bf 264} (2014), 296--307.

\bibitem[Ha\,1977]{Ha}
{R.\;Hartshorne}, {\it Algebraic Geometry}, Graduate Texts in Mathematics, Vol. 52, Springer-Verlag, New York, 1977.

\bibitem[Ka\,2004]{Ka2} S.\;Kaliman, {\it Free ${\mathbb C}_+$-actions on ${\mathbb C}^3$ are translations}, Invent. Math. {\bf 156} (2004), no. 1,
163--173.

\bibitem[Ka\,2012]{Ka} S.\;Kaliman, {\it Flexible varieties and automorphism groups}, Internat. conference ``{\it Birational\;and\;Affine\;Geometry}'', Steklov Math. Inst. of RAS, Moscow, Ap\-ril 23--27, 2012,
    {\tt http$:\!\!/\!\!/$www.mathnet.ru/php/presentation.phtml?presentid=}
    {\tt 4853$\&$option$\underline{\ }$lang=eng}.	

    \bibitem[L\"u\,1876]{Lu} J.\;L\"uroth, {\it Beweis eines Satzes tiber rationale Kurven},
Math. Annalen {\bf 9} (1876), 163--165.


 \bibitem[Ma\,1963]{Ma63}
{\rm H.\;Matsumura}, {\it On algebraic groups of birational transformations}, Atti Accad.
Naz. Lincei Rend. Cl. Sci. Fis. Mat. Natur. (8) {\bf 34} (1963), 151--155.

\bibitem[Po\,1987]{P87} V.\;L.\;Popov, {\it On actions of ${\mathbf G}_a$ on ${\mathbf A}^n$}, in: {\it Algebraic Groups} (Utrecht, 1986), Lecture Notes in Mathematics, Vol. 1271, Springer, Berlin, 1987, 237--242.

\bibitem[Po\,$2013_1$]{P13} V.\;L.\;Popov, {\it Some subgroups of the Cremona groups}, in:\;{\it Affine Algebraic Geometry}, Proceedings of the conference on the occasion of M. Miyanishi's 70th birthday (Osaka, Japan, 3--6 March 2011), World Scientific Publishing, Singapore, 2013, pp.\;213--242.



\bibitem[Po\,$2013_2$]{P13-2} V.\;L.\;Popov, {\it Tori in the Cremona groups}, Izv. Math.
{\bf 77} (2013), no. 4, 742--771.

\bibitem[Po\,$2014$]{P14}    V. \;L.\;Popov, {\it On infinite dimensional algebraic transformation groups}, Transform. Groups {\bf 19} (2014), no. 2, 549--568.

    \bibitem[Po\,2015]{P15} V.\;L.\;Popov, {\it Birational splitting and algebraic group actions}, to appear in European J. of Math.,
    {\tt arXiv:1502.02167} (2015).

    \bibitem[PV\,1994]{PV94}
{\rm V.\;L.\;Popov, E.\;B.\;Vinberg}, {\em Invariant theory}, in:\,{\em Algebraic
  Geometry} IV, Encyclopaedia of Mathematical Sciences, Vol.\,55,
  Springer-Verlag, Berlin, 1994, pp. 123--284.

 \bibitem[Ra\,1964]{R64}    C.\;P.\;Ramanujam, {\it A note on automorphism groups of algebraic varieties},
Math. Annalen {\bf 156} (1964), 25--33.

\bibitem[Re\,1968]{Re} R.\;Rentschler, {\it Op\'erations du groupe additif sur le plan},
C. R. Acad. Sci. Paris {\bf 267} (1968), 384--387.

      \bibitem[Ro\,1956]{Ro56} M.\;Rosenlicht, {\it Some basic theorems on algebraic groups},
Amer. J. Math. {\bf 78} (1956), 401--443.

   \bibitem[Ro\,1961]{Ro61} M.\;Rosenlicht, {\it On quotient varieties and the affine embedding of certain homogeneous spaces},
       Trans. Amer. Math. Soc. {\bf 101} (1961), 211--223.

           \bibitem[Ro\;1963]{Ro63} {\rm M.\;Rosenlicht}, {\it A remark on quotient spaces}, Anais Acad. Brasil. Ci\^enc. {\bf 35} (1963), 487--489

 \bibitem[S-B\,2004]{S-B04} N.-I. Shepherd-Barron, {\it
   Stably rational irrational varieties}, in: {\it The Fano Confe\-ren\-ce},  Univ. Torino, Turin, 2004, pp.\;693--700.

\bibitem[Sp\,1998]{Spr} {\rm T.\;A.\;Springer}, {\it Linear Algebraic Groups}, 2nd ed.,
 Birkh\"auser, Boston, 1998.

\bibitem[vdW\,1971]{vdW}    B.\;L.\;van der Waerden, {\it Algebra {\rm I}}, Springer-Verlag, Berlin, 1971.

\bibitem[We\,1955]{We55} {A.\;Weil}, {\it On algebraic groups of transformations}, Amer. J. Math. {\bf 77} (1955), no.\,2, 355--391.


\end{thebibliography}
\end{document}